\documentclass[preprint]{elsarticle}

\ifx\pdfoutput\@undefined\usepackage[usenames,dvips]{color}
\else\usepackage[usenames,dvipsnames]{color}
\IfFileExists{pdfcolmk.sty}{\usepackage{pdfcolmk}}{} 
\fi

\definecolor{ChadDarkBlue}{rgb}{.1,0,.2}  
\definecolor{ChadBlue}{rgb}{.1,.1,.5}  
\definecolor{ChadRoyal}{rgb}{.2,.2,.8}  
\definecolor{ChadGreen}{rgb}{0,.4,0}    
\definecolor{ChadRed}{rgb}{.5,0,.5}  

\def\oprava#1{{#1}} 
\def\zmena#1{{#1}} 
\def\autori#1{{#1}} 

\usepackage{amssymb}
\usepackage{amsmath}
\usepackage{amsthm}
\usepackage{epic,eepic}
\usepackage[mathscr]{euscript}

\usepackage{graphicx}

\usepackage{lineno}
\usepackage{comment}

\def\smallskip{\vskip\smallskipamount}
\def\medskip{\vskip\medskipamount}
\def\bigskip{\vskip\bigskipamount}

\newtheorem{theorem}{Theorem}[section]
\newtheorem{lemma}[theorem]{Lemma}
\newtheorem{statement}[theorem]{Statement}
\newtheorem{proposition}[theorem]{Proposition}
\newtheorem{corollary}[theorem]{Corollary}
\newtheorem{remark}[theorem]{Remark}

\newtheorem{definition}[theorem]{Definition}

\def\operatorname#1{{\mathop{\rm #1}}}
\newcommand{\ord}{\operatorname{ord}}

\def\comp{\leftrightarrow}

\newcommand\implik{{\ \Longrightarrow\ }}

\newcounter{ok}
{\end{list}}

\newcounter{aok}
{\end{list}}

\def\go#1;#2;#3 {\vbox to0pt{\kern-#3\hbox{\kern#2 #1}\vss}\nointerlineskip}

\newcommand{\Mea}{{\text{\rm{}M}}}

\newcommand{\Tea}{{\text{\rm{}T}}}
\newcommand{\B}{\text{\rm{}B}}
\newcommand{\C}{\text{\rm{}C}}
\newcommand{\Sh}{\text{\rm{}S}}

\begin{document}

\title{More about sharp and meager elements in 
Archimedean atomic lattice effect algebras}

\author[jp]{Josef~Niederle}
\ead{niederle@math.muni.cz}

\author[jp]{Jan~Paseka\corref{cor1}} 
\ead{paseka@math.muni.cz}
\address[jp]{Department of Mathematics and Statistics,
			Faculty of Science,
			Masaryk University,
			{Kotl\'a\v r{}sk\' a\ 2},
			CZ-611~37~Brno, Czech Republic}        	

\cortext[cor1]{Corresponding author}

\date{Received: date / Accepted: date}


\begin{abstract}
The aim of our paper is twofold. First, 
we thoroughly study the set of meager elements $\Mea(E)$, 
the center $\C(E)$ and the compatibility 
center $\B(E)$ in the setting of 
atomic Archimedean lattice effect algebras $E$. The main result is that 
in this case  the center $\C(E)$ is bifull (atomic) iff 
the compatibility center $\B(E)$ is bifull (atomic) whenever $E$ is 
sharply dominating. As a by-product, we 
give a new descriciption  of the smallest sharp element over $x\in E$ via the 
basic decomposition of $x$. 
Second, we prove the Triple Representation Theorem for 
sharply dominating atomic Archimedean lattice effect algebras. 

\end{abstract}


\maketitle

\zmena{\section*{Introduction}}

\label{intro}
The history of quantum structures started at the beginning of the 20th century.
Observable events constitute a Boolean algebra in a classical physical system.
Because event structures in quantum mechanics cannot be described by Boolean algebras, 
Birkhoff and von Neumann introduced orthomodular lattices
which were considered as the standard quantum logic. Later on, orthoalgebras  were introduced
as the generalizations of orthomodular posets, which were considered
as "sharp"  quantum logic.

In the nineties of the twentieth century, two
equivalent quantum structures, D-posets and effect algebras were extensively
studied, which were considered as "unsharp" generalizations of the
structures which arise in quantum mechanics, in particular, of orthomodular
lattices and MV-algebras.

In \cite{PR6x} Paseka and Rie\v canov\'a published as open problem whether the center
$\C(E)$ is a bifull sublattice of an Archimedean atomic lattice effect algebra $E$. This 
question was answered by M. Kalina in \cite{kalina1} who proved 
that $\C(E)$ need not be 
a bifull sublattice of $E$ even if $C(E)$ is atomic. 

The aim of our paper is twofold. First, 
we thoroughly study the set of meager elements $\Mea(E)$, 
the center $\C(E)$ and the compatibility 
center $\B(E)$ in the setting of 
atomic Archimedean lattice effect algebras $E$. The main result 
of Section 2 is that 
in this case  the center $\C(E)$ is bifull (atomic) iff 
the compatibility center $\B(E)$ is bifull (atomic) whenever $E$ is 
sharply dominating. As a by-product, we 
give a new descriciption  of the smallest sharp element over $x\in E$ via the 
basic decomposition of $x$. 
Second, in Section 3  we prove the Triple Representation Theorem established 
by G. Jen\v{c}a in \cite{jenca} in the setting of complete lattice effect 
algebras  for sharply dominating atomic Archimedean lattice effect algebras.

\medskip

\section{{Preliminaries} and basic facts}

\zmena{Effect algebras were introduced by D.J. Foulis and 
M.K. Bennett (see \cite{FoBe}) 
for modelling unsharp measurements in a Hilbert space. In this case the 
set $E(H)$ of effects is the set of all self-adjoint operators $A$ on 
a Hilbert space $H$ between the null operator $0$ and the identity 
operator $1$ and endowed with the partial operation $+$ defined 
iff  $A+B$ is in $E(H)$, where $+$ is the usual operator sum. }

\zmena{In general form, an effect algebra is in fact a partial algebra 
with one partial binary operation and two unary operations satisfying 
the following axioms due to D.J. Foulis and 
M.K. Bennett.}

\begin{definition}\label{def:EA}{\autori{{\rm{}\cite{ZR62}  }}
{\rm A partial algebra $(E;\oplus,0,1)$ is called an {\em effect algebra} if
$0$, $1$ are two distinct elements and $\oplus$ is a partially
defined binary operation on $E$ which satisfy the following
conditions for any $x,y,z\in E$:
\begin{description}
\item[\rm(Ei)\phantom{ii}] $x\oplus y=y\oplus x$ if $x\oplus y$ is defined,
\item[\rm(Eii)\phantom{i}] $(x\oplus y)\oplus z=x\oplus(y\oplus z)$  if one
side is defined,
\item[\rm(Eiii)] for every $x\in E$ there exists a unique $y\in
E$ such that $x\oplus y=1$ (we put $x'=y$),
\item[\rm(Eiv)\phantom{i}] if $1\oplus x$ is defined then $x=0$.
\end{description}
}%
}
\end{definition}

We often denote the effect algebra $(E;\oplus,0,1)$ briefly by
$E$. On every effect algebra $E$  a partial order
$\le$  and a partial binary operation $\ominus$ can be 
introduced as follows:
\begin{center}
$x\le y$ \mbox{ and } {\autori{$y\ominus x=z$}} \mbox{ iff }$x\oplus z$
\mbox{ is defined and }$x\oplus z=y$\,.
\end{center}

If $E$ with the defined partial order is a lattice (a complete
lattice) then $(E;\oplus,0,1)$ is called a {\em lattice effect
algebra} ({\em a complete lattice effect algebra}).

\begin{definition}\label{subef}{\rm
Let $E$ be an  effect algebra.
Then $Q\subseteq E$ is called a {\em sub-effect algebra} of  $E$ if 
\begin{enumerate}
\item[(i)] $1\in Q$
\item[(ii)] if out of elements $x,y,z\in E$ with $x\oplus y=z$
two are in $Q$, then $x,y,z\in Q$.
\end{enumerate}
If $E$ is a lattice effect algebra and $Q$ is a sub-lattice and a sub-effect
algebra of $E$, then $Q$ is called a {\em sub-lattice effect algebra} of $E$.}
\end{definition}

Note that a sub-effect algebra $Q$ 
(sub-lattice effect algebra $Q$) of an  effect algebra $E$ 
(of a lattice effect algebra $E$) with inherited operation 
$\oplus$ is an  effect algebra (lattice effect algebra) 
in its own right.


For an element $x$ of an effect algebra $E$ we write
$\ord(x)=\infty$ if $nx=x\oplus x\oplus\dots\oplus x$ ($n$-times)
exists for every positive integer $n$ and we write $\ord(x)=n_x$
if $n_x$ is the greatest positive integer such that $n_xx$
exists in $E$.  An effect algebra $E$ is {\em Archimedean} if
$\ord(x)<\infty$ for all $x\in E$.

A minimal nonzero element of an effect algebra  $E$
is called an {\em atom}  and $E$ is
called {\em atomic} if below every nonzero element of 
$E$ there is an atom.

For a poset $P$ and its subposet $Q\subseteq P$ we denote, 
for all $X\subseteq Q$, by $\bigvee_{Q} X$ the join of 
the subset $X$ in the poset $Q$ whenever it exists. 
Recall also $Q\subseteq P$ is {\em densely embedded in} $P$ if for every
element $x\in P$ there exist $S,T\subseteq Q$ such that
$x=\bigvee_{{P}} S=\bigwedge_{{P}} T$. 

We say that a finite system $F=(x_k)_{k=1}^n$ of not necessarily
different elements of an effect algebra $(E;\oplus,0,1)$ is
\zmena{\it orthogonal} if $x_1\oplus x_2\oplus \cdots\oplus
x_n$ (written $\bigoplus\limits_{k=1}^n x_k$ or $\bigoplus F$) exists
in $E$. Here we define $x_1\oplus x_2\oplus \cdots\oplus x_n=
(x_1\oplus x_2\oplus \cdots\oplus x_{n-1})\oplus x_n$ supposing
that $\bigoplus\limits_{k=1}^{n-1}x_k$ is defined and
$\bigoplus\limits_{k=1}^{n-1}x_k\le x'_n$. We also define 
$\bigoplus \emptyset=0$.
An arbitrary system
$G=(x_{\kappa})_{\kappa\in H}$ of not necessarily different
elements of $E$ is called \zmena{orthogonal} if $\bigoplus K$
exists for every finite $K\subseteq G$. We say that for a \zmena{orthogonal} 
system $G=(x_{\kappa})_{\kappa\in H}$ the
element $\bigoplus G$ exists iff
$\bigvee\{\bigoplus K
\mid
K\subseteq G$ is finite$\}$ exists in $E$ and then we put
$\bigoplus G=\bigvee\{\bigoplus K\mid K\subseteq G$ is
finite$\}$. We say that $\bigoplus G$ is the {\em orthogonal sum} 
of $G$. (Here we write $G_1\subseteq G$ iff there is
$H_1\subseteq H$ such that $G_1=(x_{\kappa})_{\kappa\in
H_1}$). 

An element $u\in E$ is called {\em finite\/}
if either $u=0$ or there is a finite sequence $\{a_1,a_2,\dots,a_n\}$ of not
necessarily different atoms of $E$ such that $u=a_1\oplus
a_2\oplus \dots\oplus a_n$. Note that any atom of $E$ is evidently finite. 
An element $v\in E$ is called {\em cofinite\/}
if $v'\in E$ is finite.

\oprava{Elements $x$ and $y$ of a lattice effect algebra $E$ are
called {\em compatible} ($x\leftrightarrow y$ for short) if 
$x\vee y=x\oplus(y\ominus(x\wedge y))$ 
(see \cite{Kop2,ZR56}).}

Remarkable sub-lattice effect algebras of 
a lattice effect algebra $E$ are
\begin{enumerate}
\item[(1)] A {\em block} $M$ of $E$, which is any maximal subset of pairwise compatible
elements of $E$ (in fact $M$ is a maximal sub-$MV$-algebra of $E$, see 
\cite{ZR56}).

\item[(2)] The {\em set} $\Sh(E)=\{x\in E\mid x\wedge x'=0\}$ 
{\em of sharp elements  of}  $E$ (see \cite{gudder1}, \cite{gudder2}), 
which is an orthomodular lattice (see \cite{ZR57}).

\item[(3)] The {\em compatibility center} $\B(E)$ of $E$, 
 $\B(E)=\bigcap\{M\subseteq E\mid M \ 
\mbox{is a block}$\ $\mbox{of}\ E\}=\{x\in E\mid x\comp y\ 
\mbox{for every}$ $y\in E\}$ which is in fact 
an $MV$-algebra ($MV$-effect algebra).
\item[(4)]  The {\em center} 
$\C(E)=\{ x\in E \mid y=(y\wedge x)\vee(y\wedge x')\ 
\mbox{for all}\ y\in E\}$ of $E$  is a Boolean algebra
(see \cite{GrFoPu}). In every lattice effect algebra 
it holds  $\C(E)=\B(E)\cap \Sh(E)=\Sh(\B(E))$ (see \cite{ZR51} and \cite{ZR52}).
\end{enumerate}

All these sub-lattice effect algebras of a lattice effect algebra $E$ 
are in fact  {\em full sub-lattice effect algebras} of $E$. 
This means that they are closed with respect to all suprema and infima 
existing in $E$ of their subsets \cite{ZR57,ZR60}. 

The  {\em $MV$-effect algebras} $E$ are precisely lattice 
effect algebras with a unique block (i.e., $E=\B(E)$).

The following statements are well known.

\begin{statement}\label{lem:2.1}
Let $E$ be a lattice effect algebra. 
Then
\begin{enumerate}
\item[\rm(i)]  {\rm \cite[Theorem 2.1]{ZR57}}\/ Assume $b\in E$,
$A\subseteq E$ are such that $\bigvee A$ exists in $E$ and $b\comp a$
for all $a\in A$. Then
\begin{enumerate}
\item[{\rm(a)}]
$b\comp \bigvee A$.
\item[{\rm(b)}]
$\bigvee\{b\wedge a: a\in A\}$ exists in $E$ and
equals $b\wedge(\bigvee A)$.
\end{enumerate}

\item[\rm(ii)] {\rm \cite[Theorem 3.7]{ZR57}, \cite[Theorem 2.8]{ZR60}}\/  %
$\Sh(E)$, $\B(E)$  and $\C(E)$
are full sub-lattice effect algebras of $E$.

\item[\rm(iii)] {\rm\cite[Lemma 3.3]{PR5}} Let  $x, y\in E$. 
Then $x\wedge y=0$ and $x\leq y'$ iff $kx\wedge ly=0$ and $kx\leq (ly)'$, 
whenever $kx$ and $ly$ exist in $E$.

\item[\rm(iv)] {\rm\cite[Proposition 1]{PR7x}} Let 
$\{b_{\alpha} \mid \alpha\in \Lambda\}$ 
be a family of elements in  $E$ and let 
$a\in E$ with $a\leq b_{\alpha}$ for all $\alpha\in \Lambda$. Then
$$
(\bigvee\{b_{\alpha} \mid \alpha\in \Lambda\})\ominus a=%
\bigvee\{b_{\alpha}\ominus a \mid \alpha\in \Lambda\}
$$
if one side is defined.

\item[\rm(v)] {\rm\cite[Theorem 3.5]{wujunde}} For every atom 
$a\in E$ with $ord(a) < \infty$, $n_aa$ is
the smallest sharp element over $a$.

\item [\rm(vi)] {\rm\cite[Corollary 4.3]{ZR52}} Let $x, y\in E$. 
Then $x\oplus y=(x\vee y)\oplus (x\wedge y)$ whenever  $x\oplus y$ exists.

\item [\rm(vii)] {\rm\cite[Proposition 1.8.7]{dvurec}} 
Let  $b\in E$, $A\subseteq E$ are such that $\bigvee A$ exists in $E$ and $b\oplus a$ 
exists for all $a\in A$. Then 
$\bigvee\{b\oplus a: a\in A\}=b\oplus \bigvee A$.

\item [\rm(viii)] {\rm\cite[Lemma 4.1]{ZR70}} Assume that 
$z\in \C(E)$. Then,  for 
all $x, y\in E$ with $x\leq y'$,  
$(x\oplus y)\wedge z = (x\wedge z)\oplus (y\wedge z)$.
\end{enumerate}
\end{statement}

\begin{statement}\label{lem:2.0}{\rm \cite[Theorem 3.3]{ZR65}} 
Let $E$ be an Archimedean atomic lattice effect
algebra. Then
to every nonzero
element $x\in E$ there are mutually distinct atoms $a_{\alpha}\in
E$ and positive integers $k_{\alpha}$, $\alpha\in{\cal E}$ such
that
$$
x=\bigoplus\{k_{\alpha}a_{\alpha}\mid
\alpha\in{\cal E}\}=
\bigvee\{k_{\alpha}a_{\alpha}\mid
\alpha\in{\cal E}\},
$$
and $x\in S(E)$ iff
$k_{\alpha}=n_{a_{\alpha}}=\ord(a_{\alpha})$ for all
$\alpha\in\cal E$.
\end{statement}

\begin{statement}\label{lem:2.2}{\rm\cite[Theorem 8]{mosna}} 
Let $E$ be an atomic Archimedean lattice effect
algebra and let $\mathcal{M}=\{M_\kappa|\kappa\in H\}$ be a family
of all atomic blocks of $E$. For each $\kappa\in H$ let $A_\kappa$
be the set of all atoms of $M_\kappa$. Then:
\begin{enumerate}
\item[\rm(i)] 
For each $\kappa\in H$, $A_\kappa$ is a maximal pairwise
compatible set of atoms of $E$.
\item[\rm(ii)]For $x\in E$ and $\kappa\in H$ it holds $x\in M_\kappa$
iff $x\leftrightarrow A_\kappa$.
\item[\rm(iii)] $M\in\mathcal{M}$ iff there exists a maximal pairwise
compatible set $A$ of atoms of $E$ such that $A\subseteq M$ and if
$M_1$ is a block of $E$ with $A\subseteq M_1$ then $M=M_1$.
\item[\rm(iv)] $E=\bigcup\{M_\kappa|\kappa\in H\}$.
\item[\rm(v)] $B(E)=\bigcap\{M_\kappa|\kappa\in H\}$.
\item[\rm(vi)] $C(E)=\bigcap\{C(M_\kappa)|\kappa\in H\}=\bigcap\{S(M_\kappa)|\kappa\in H\}$.
\item[\rm(vii)] $S(E)=\bigcup\{C(M_\kappa)|\kappa\in H\}=\bigcup\{S(M_\kappa)|\kappa\in H\}$.
\end{enumerate}
\end{statement}

\begin{lemma}\label{joyka}
Let $E$ be a lattice effect algebra and let 
$b\in E$, $A\subseteq E$ are such that $\bigvee A$ exists in $E$ and $b\oplus a$ 
exists for all $a\in A$. Then 
$b\oplus \bigvee A$ exists in $E$ and 
$b\oplus \bigvee A=(b\vee \bigvee A)\oplus \bigvee\{b\wedge a: a\in A\}$.
\end{lemma}
\begin{proof} Clearly $b\comp a$
for all $a\in A$. By Statement \ref{lem:2.1}, (i) we have that 
$b\comp \bigvee A$ and $\bigvee\{b\wedge a: a\in A\}=b\wedge(\bigvee A)$. 
Furthermore, $b\leq a'$ for all  $a\in A$ and hence 
$b\leq \bigwedge\{a' \mid a\in A\}= (\bigvee A)'$. Therefore 
$b\oplus \bigvee A$ exists. In view of Statement \ref{lem:2.1}, (vi) 
$$
b\oplus \bigvee A=(b\vee \bigvee A)\oplus (b\wedge \bigvee A)=%
(b\vee \bigvee A)\oplus \bigvee\{b\wedge a: a\in A\}.
$$
\end{proof}

\section{Bifull sub-lattice effect algebras of  lattice effect algebras}
\label{Atomicity}

\begin{definition}\label{bifdf}{\rm
For a poset $L$ and a subset $D\subseteq L$ we say that 
$D$ is a {\em  $\bigvee$-bifull sub-poset of} $L$  iff, 
for any $X\subseteq D$ , $\bigvee_L X$ exists iff \, 
$\bigvee_D X$ exists, in which case $\bigvee_L X=\bigvee_D X$. 
Dually, the notion of {\em  $\bigwedge$-bifull sub-poset of}  $L$  
is defined. We call a subset $D\subseteq L$ to be 
a {\em  bifull sub-poset of} $L$  if it is both $\bigvee$-bifull and 
$\bigwedge$-bifull.}
\end{definition}

\begin{remark}\label{combif}{\rm Clearly, if 
$L$ is a complete lattice  then $D\subseteq L$ is
a complete sub-lattice of $L$ (i.e., $D$ inherits all suprema 
and infima of its subsets existing in $L$) iff 
$D$ is a bifull sub-poset of $L$. Moreover, if 
$E$  is a lattice effect algebra then a 
sub-lattice effect algebra $D$ of $E$ is a 
{\em bifull sub-lattice effect algebra} of $E$ iff it 
is $\bigvee$-bifull.
}
\end{remark}

An important class of effect algebras 
was introduced by S. Gudder in \cite{gudder1} 
and \cite{gudder2}. Fundamental example is the 
standard Hilbert spaces effect algebra ${\cal E}({\mathcal H})$.

For an element $x$ of an effect algebra $E$ we denote
$$
\begin{array}{r c  l c l}
\widetilde{x}&=\bigvee_{E}\{s\in \Sh(E) \mid s\leq x\}&%
\phantom{xxxxx}&\text{if it exists and belongs to}\ \Sh(E)\phantom{.}\\
\widehat{x}&=\bigwedge_{E}\{s\in \Sh(E) \mid s\geq x\}&\phantom{xxxxx}&%
\text{if it exists and belongs to}\ \Sh(E).
\end{array}
$$

\begin{definition} {\rm (\cite{gudder1},  \cite{gudder2}.) 
An effect algebra $(E, \oplus, 0,
1)$ is called {\em sharply dominating} if for every $x\in E$ there
exists $\widehat{x}$, the smallest sharp element  such that $x\leq
\widehat{x}$. That is $\widehat{x}\in \Sh(E)$ and if $y\in \Sh(E)$ satisfies
$x\leq y$ then $\widehat{x}\leq y$. }
\end{definition}

Recall that evidently an effect algebra $E$ is sharply dominating iff 
for every $x\in E$ there exists $\widetilde{x}\in \Sh(E)$ such
that $\widetilde{x}\leq x$ and if $u\in \Sh(E)$
satisfies $u\leq x$ then $u\leq \widetilde{x}$ iff for every $x\in E$ there exist 
a smallest sharp element $\widehat{x}$ over $x$ and a greatest sharp 
element $\widetilde{x}$ below $x$.

In what follows set (see \cite{jenca,wujunde})
$$\Mea(E)=\{x\in E \mid\ \text{if}\ v\in \Sh(E)\ \text{satisfies}\ v\leq x\ 
\text{then}\ v=0\}.$$

An element $x\in \Mea(E)$ is called {\em meager}. Moreover, $x\in \Mea(E)$ 
iff $\widetilde{x}=0$. Recall that $x\in \Mea(E)$, $y\in E$, $y\leq x$ implies 
$y\in \Mea(E)$ and $x\ominus y\in \Mea(E)$.

{\renewcommand{\labelenumi}{{\normalfont  (\roman{enumi})}}
\begin{lemma}\label{jmpy2} Let $E$ be an effect algebra 
in which $\Sh(E)$ is a sub-effect algebra of $E$ and 
let $x\in \Mea(E)$ such that $\widehat{x}$  exists. Then 
\begin{enumerate}
\settowidth{\leftmargin}{(iiiii)}
\settowidth{\labelwidth}{(iii)}
\settowidth{\itemindent}{(ii)}
\item $\widehat{x}\ominus x\in \Mea(E)$.
\item If $y\in \Mea(E)$ such that $x\oplus y$ %
exists and $x\oplus y=z\in\Sh(E)$ then  $\widehat{x}=z$. 
Moreover, if $E$ is a lattice effect algebra then $\widehat{y}$ 
exists and $\widehat{y}=\widehat{\widehat{x}\ominus x}=z$.
\end{enumerate}
\end{lemma}}
\begin{proof} (i): Let $u\in\Sh(E)$ such that $u\leq \widehat{x}\ominus x$. 
Then $x\leq \widehat{x}\ominus u\in\Sh(E)$ which yields that 
$\widehat{x}\leq \widehat{x}\ominus u$. Hence $u=0$, i.e., 
$\widehat{x}\ominus x\in \Mea(E)$.

\noindent{}(ii): Since $x\leq z$ and hence $\widehat{x}\leq z$ we have 
$x\oplus y=z=\widehat{x}\oplus (z\ominus \widehat{x})$ and 
$\widehat{x}=x\oplus (\widehat{x}\ominus x)$. This yields 
$x\oplus y=x\oplus (\widehat{x}\ominus x)\oplus (z\ominus \widehat{x})$. 
By the cancellation law we get 
$y=(\widehat{x}\ominus x)\oplus %
\underbrace{(z\ominus \widehat{x})}_{\in\Sh(E)}$. Hence 
$z\ominus \widehat{x}=0$, i.e., $z=\widehat{x}$.

Now, assume that $E$ is a lattice effect algebra. Let $u\in\Sh(E)$, 
$u\geq y$. Then also $u\wedge z\geq y$, $u\wedge z\in \Sh(E)$ and 
$z\ominus (u\wedge z)\in \Sh(E)$. Then 
$x\oplus y=z=y \oplus ((u\wedge z)\ominus y)\oplus (z\ominus (u\wedge z))$. 
Therefore $x=((u\wedge z)\ominus y)\oplus %
(z\ominus (u\wedge z))$. 
Since $z\ominus (u\wedge z)\in\Sh(E)$ this 
yields $z\ominus (u\wedge z)=0$, i.e., 
$z=u\wedge z\leq u$.
\end{proof}

{\renewcommand{\labelenumi}{{\normalfont  (\roman{enumi})}}
\begin{lemma}\label{jpy2} Let $E$ be an effect algebra 
in which $\Sh(E)$ is a sub-effect algebra of $E$ and 
let $x\in E$ such that $\widetilde{x}$  
exists. Then 
\begin{enumerate}
\settowidth{\leftmargin}{(iiiii)}
\settowidth{\labelwidth}{(iii)}
\settowidth{\itemindent}{(ii)}
\item  ${x\ominus \widetilde{x}}\in  \Mea(E)$ and 
$x=\widetilde{x}\oplus {(x\ominus \widetilde{x})}$ is 
the unique decomposition $x = x_S \oplus x_M$, 
where $x_S\in\Sh(E)$ and $x_M \in \Mea(E)$. Moreover, 
$x_S\wedge x_M=0$ and if $E$ is a lattice effect algebra then 
 $x = x_S \vee x_M$. 
\item If $E$ is a lattice effect algebra such that $\widehat{x}$ exists then 
$\widehat{x\ominus \widetilde{x}}$ and  $\widehat{\widehat{x}\ominus {x}}$ exist,  
$\widehat{x}\ominus \widetilde{x}=\widehat{x\ominus \widetilde{x}}=\widehat{\widehat{x}\ominus {x}}$,
$\widehat{x}=\widetilde{x}\oplus \widehat{x\ominus \widetilde{x}}=%
\widetilde{x}\vee \widehat{x\ominus \widetilde{x}}$  and 
$\widetilde{x}\wedge \widehat{x\ominus \widetilde{x}}=0$. 
Moreover, $\widehat{x}\ominus \widetilde{x}=\widehat{(\widehat{x}\ominus \widetilde{x})\ominus %
(\widehat{x}\ominus {x})}$.
\end{enumerate}
\end{lemma}}
\begin{proof} (i): Let $v\in \Sh(E)$, $v\leq x\ominus \widetilde{x}$. Then 
$v\oplus \widetilde{x}\leq x$ and $v\oplus \widetilde{x}\in \Sh(E)$. 
Hence $v\oplus \widetilde{x}\leq \widetilde{x}$, i.e., $v=0$, 
$x_M \in \Mea(E)$ and 
$x=\widetilde{x}\oplus ({x\ominus \widetilde{x}})$. 

Assume that there is a decomposition $x = x_S \oplus x_M$ such that  
$x_S\in\Sh(E)$ and $x_M \in \Mea(E)$. Then $x_S\leq \widetilde{x}$ and 
$x_M=x \ominus x_S= %
(x \ominus \widetilde{x})\oplus  (\widetilde{x} \ominus x_S)\geq %
\widetilde{x} \ominus x_S\in \Sh(E)$. It follows that 
$\widetilde{x} \ominus x_S=0$ because $x_M \in \Mea(E)$. 
Therefore, $x_M=x \ominus \widetilde{x}$ and $x_S=\widetilde{x}$.

We have $x^{'}_S=1\ominus x_S\geq x \ominus x_S= x_M$. 
Hence $x_S\wedge x_M\leq x_S\wedge x^{'}_S=0$.

Let $E$ be a lattice effect algebra. By Statement
\ref{lem:2.1}, (vi) we have that 
$x = x_S \oplus x_M= (x_S \vee x_M)\oplus (x_S \wedge x_M)=x_S \vee x_M$.

\noindent{}(ii): We have that $\widehat{x}\ominus \widetilde{x}\geq 
x\ominus \widetilde{x}$, $\widehat{x}\ominus \widetilde{x}\in \Sh(E)$. 
Let $z\in \Sh(E)$ such that 
$z\geq x\ominus \widetilde{x}$. Let us 
put $w=z\wedge (\widehat{x}\ominus \widetilde{x})$. Then 
$w\leq \widehat{x}\ominus \widetilde{x}$ hence 
$w\oplus  \widetilde{x}\in \Sh(E)$ exists and 
$w\oplus  \widetilde{x}\geq x$. This yields that 
$w\oplus  \widetilde{x}\geq \widehat{x}$, i.e., 
$z\geq w\geq \widehat{x}\ominus  \widetilde{x}$.
Therefore, $\widehat{x}\ominus  \widetilde{x}=\widehat{x\ominus \widetilde{x}}$.

Since $\widehat{x}\ominus \widetilde{x}\leq 1\ominus \widetilde{x}=(\widetilde{x})'$ we obtain that 
$\widetilde{x}\wedge \widehat{x\ominus \widetilde{x}}\leq \widetilde{x}\wedge (\widetilde{x})'=0$.

We proceed similarly to prove that 
$\widehat{x}\ominus \widetilde{x}=\widehat{\widehat{x}\ominus {x}}$. 
Evidently, $\widehat{x}\ominus \widetilde{x}\geq {\widehat{x}\ominus {x}}$. 
Let $z\in \Sh(E)$ such that $z\geq {\widehat{x}\ominus {x}}$. We put 
$w=z\wedge (\widehat{x}\ominus \widetilde{x})$. Then 
${\widehat{x}\ominus {x}}\leq w\leq \widehat{x}\ominus \widetilde{x}\leq \widehat{x}$. 
It follows that ${\widehat{x}\ominus w \leq x}$ and ${\widehat{x}\ominus w \in \Sh(E)}$ 
which yields that 
${\widehat{x}\ominus w \leq \widetilde{x}}$. Hence 
${\widehat{x}\ominus \widetilde{x}}\leq w\leq z$.

Moreover, $\widehat{(\widehat{x}\ominus \widetilde{x})\ominus %
(\widehat{x}\ominus {x})}=\widehat{x\ominus \widetilde{x}}=\widehat{x}\ominus \widetilde{x}$.

\end{proof}

As proved in \cite{cattaneo}, 
$\Sh(E)$ is always a sub-effect algebra in 
a sharply dominating  effect algebra $E$.

\begin{corollary}\label{gejza}{\rm{}\cite[Proposition 15]{jenca}}
Let $E$ be a sharply dominating  effect algebra. 
Then every $x \in E$ has a
unique decomposition $x = x_S \oplus x_M$, where $x_S\in\Sh(E)$ and $x_M \in \Mea(E)$, 
namely $x=\widetilde{x}\oplus {(x\ominus \widetilde{x})}$.
\end{corollary}

Moreover, the following statement holds. 

{\renewcommand{\labelenumi}{{\normalfont  (\roman{enumi})}}
\begin{statement}\label{shdisbifullcv} Let $E$ be a lattice effect algebra. Then 
\begin{enumerate}
\settowidth{\leftmargin}{(iiiii)}
\settowidth{\labelwidth}{(iii)}
\settowidth{\itemindent}{(ii)}
\item  {\rm \cite[Corollary 1]{kaparie}} If 
$E$ is a sharply dominating then $S(E)$ is bifull in $E$.
 \item {\rm \cite[Lemma 2.7]{PR6x}}   If $E$ is Archimedean and atomic  then 
$S(E)$ is bifull in $E$.
\end{enumerate}
\end{statement}}

First,  we shall need an extension of Statement
\ref{lem:2.1}, (iii).

\begin{lemma}\label{extfu} Let $E$ be a lattice effect algebra, 
$x_1, \dots, x_n\in E$, 
$k_1, \dots, k_n\in {\mathbb N}$, $n\geq 2$ such that 
$k_ix_i$ exist in $E$ for all\, $1\leq i \leq n$. 
Then 
\begin{center}
\begin{tabular}{r l}
&$x_i\wedge x_j=0$ and $x_i\leq x_j'$ for all $1\leq i < j \leq n$\\
 iff& \\ 
\phantom{\huge I}&$\bigoplus_{j=1}^{n} k_j x_j$ exists and\  $\bigoplus_{j=1}^{n} k_j x_j= \bigvee_{j=1}^{n} k_j x_j$,\\
\phantom{\huge I}&$\bigoplus_{i\in I}k_i x_i\wedge \bigoplus_{j\in J} k_j x_j=0$ and 
\ $\bigoplus_{j\in J} k_j x_j \leq (\bigoplus_{i\in I} k_i x_i)'$\\ 
\phantom{\huge I}&for all\, $\emptyset \not= I \subset \{1, \dots, n\}$, $J=\{1, \dots, n\}\setminus I$. 
\end{tabular}
\end{center}
\end{lemma}
\begin{proof} Assume that $x_i\wedge x_j=0$ and $x_i\leq x_j'$ for all $1\leq i < j \leq n$. 
Let $k_ix_i$ exist in $E$ for all $1\leq i \leq n$.
If $n=2$ then from 
Statement \ref{lem:2.1}, (iii) we know that $k_1 x_1\wedge k_2 x_2=0$, 
$k_1 x_1\leq (k_2 x_2)'$ and $k_2 x_2\leq (k_1 x_1)'$. Since 
$k_1 x_1\comp k_2 x_2$ 
we have that $k_1 x_1\vee k_2 x_2=k_1 x_1\oplus (k_2 x_2 \ominus (k_1 x_1\wedge k_2 x_2))=%
k_1 x_1\oplus k_2 x_2$. We shall proceed by induction. 
Let $n\in {\mathbb N}$ be arbitrary, 
$n\geq 3$ and assume that the statement holds for every $m < n$. 
Let us take $\emptyset \not= I \subset \{1, \dots, n\}$ arbitrarily and put $J=\{1, \dots, n\}\setminus I$. 
Hence $|I|<n$ and $|J|<n$.
Then we have (again by Statement \ref{lem:2.1}, (iii)) 
that $k_i x_i\wedge k_j x_j=0$ and $k_j x_j\leq (k_i x_i)'$ for all $i\in I$ and $j\in J$. 
This and the induction assumption yield that 
$\bigoplus_{j\in J} k_j x_j=\bigvee_{j\in J} k_j x_j \leq (k_i x_i)'$ for all $i \in I$. 
This is equivalent to $k_i x_i \leq (\bigoplus_{j\in J} k_j x_j)'$  for all $i \in I$ i.e., 
$\bigoplus_{i\in I} k_i x_i = \bigvee_{i\in I} k_i x_i \leq (\bigoplus_{j\in J} k_j x_j)'$.
Furthermore, $k_j x_j\comp k_i x_i$ for all $i\in I$ and $j\in J$ 
implies by 
Statement \ref{lem:2.1}, (i) that 
$$\bigoplus_{i\in I} k_i x_i\wedge \bigoplus_{j\in J} k_j x_j =%
\bigvee_{i\in I} k_i x_i\wedge \bigvee _{j\in J} k_j x_j=%
\bigvee_{i\in I} \bigvee _{j\in J} (\underbrace{k_i x_i\wedge  k_j x_j}_{=0})=0.$$ 
Similarly by the induction assumption and Statement \ref{lem:2.1}, (iii) and (vii), 
$$
\begin{array}{r c l}
\bigoplus_{j=1}^{n} k_j x_j&=& \left(\bigoplus_{j=1}^{n-1} k_j x_j \right)\oplus k_n x_n  =%
\left(\bigvee_{j=1}^{n-1} k_j x_j \right)\oplus k_n x_n\\
&=& \bigvee_{j=1}^{n-1} (k_j x_j\oplus k_n x_n) %
=\bigvee_{j=1}^{n-1} (k_j x_j\vee k_n x_n)= \bigvee_{j=1}^{n} k_j x_j.
\end{array}
$$

The converse implication is evident.
\end{proof}

\begin{corollary}\label{fsumato} 
Let $E$ be an Archime\-dean lattice effect algebra and $a_1, \dots, a_n$ 
mutually compatible different atoms from $E$, $1\leq k_i \leq n_{a_i}$ for 
all $1\leq i \leq n$. Then $k_1 a_1 \oplus \dots \oplus k_n a_n$ exists and 
$k_1 a_1 \oplus \dots \oplus k_n a_n=k_1 a_1 \vee \dots \vee k_n a_n$ . 
Moreover, 
$n_{a_1} a_1 \oplus \dots \oplus n_{a_n} a_n=n_{a_1} a_1 \vee \dots \vee n_{a_n} a_n$
is the smallest sharp element over $k_1 a_1 \oplus \dots \oplus k_n a_n$.
\end{corollary}

{\renewcommand{\labelenumi}{{\normalfont  (\roman{enumi})}}
\begin{theorem}\label{popismeager} Let $E$ be 
an atomic Archime\-dean  lattice
effect algebra and let $x\in \Mea(E)$. Let us denote 
$A_x=\{a \mid \ a\ \text{an atom of}\ E,$ $ a\leq x\}$ and, for any $a\in A_x$, 
we shall put $k^{x}_a=\text{max}\{ k\in {\mathbb N}\mid ka \leq x\}$. Then
\begin{enumerate}
\settowidth{\leftmargin}{(iiiii)}
\settowidth{\labelwidth}{(iii)}
\settowidth{\itemindent}{(ii)}
\item For any $a\in A_x$ we have $k^{x}_a < n_a$.
\item The set $F_x=\{k^{x}_a a \mid \ a\in A_x\}$ is orthogonal and 
$$
x=\bigoplus\{k^{x}_a a \mid a\ \text{an atom of}\ E,\ a\leq x\}=\bigoplus F_x = \bigvee F_x.
$$
Moreover, for all $B\subseteq A_x$ and all natural numbers   $l_b< n_b, b\in B$ such that 
$x=\bigoplus \{l_b b \mid b\in B\}$ we have that $B=A_x$ and 
$l_a=k^{x}_a$ for all $a\in A_x$ i.e., $F_x$ is the unique set  
of multiples of atoms from $A_x$ such that its orthogonal sum is $x$. 
\item For every atomic block $M$ of $E$, $x\in M$ implies that 
$[0, x]\subseteq M$.
\item $x\in \B(E)$ implies that 
$[0, x]\subseteq \B(E)$.
\item If  $\widehat{x}$ exists then 
$$
\widehat{x}=\widehat{\widehat{x}\ominus x}=\bigoplus\{n_a a \mid a\ \text{an atom of}\ E,\ a\leq x\}%
=\bigvee\{n_a a \mid\ a\in A_x\}
$$
and 
$$
\widehat{x}\ominus x=\bigoplus\{(n_a-k^{x}_a) a \mid\ a\in A_x\}%
=\bigvee\{(n_a-k^{x}_a) a \mid\ a\in A_x\}.
$$

\item If $x$ is finite then $[0, x]$ is a finite lattice, 
$x=\bigoplus_{i=1}^{n} k_i a_i=\bigvee_{i=1}^{n} k_i a_i$ for a suitable finite set 
$A_x=\{a_1, \dots, a_n\}$ of atoms of $E$ and 
$[0, x]\cong \prod _{i=1}^{n} [0, k_i a_i]$.
\end{enumerate}
\end{theorem}}
\begin{proof}(i): Let $a\in A_x$. Since $E$ is Archimedean we have 
$k^{x}_a \leq n_a$. Assume that  $k^{x}_a = n_a$. 
Then $0< n_a a\leq x$ and $n_a a\in \Sh(E)$ by 
Statement \ref{lem:2.0}, i.e., $x\not\in \Mea(E)$, a contradiction.

(ii): From Statement \ref{lem:2.0}, (i) we know that there is a subset 
$B\subseteq A_x$ and natural numbers   $l_b< n_b, b\in B$ such that 
$$
x=\bigoplus \{l_b b \mid b\in B\} = \bigvee \{l_b b \mid b\in B\}.
$$
Let us show that $F_x=\{l_b b \mid b\in B\}$. 
Evidently, $l_b \leq k^{x}_b < n_b$ and $l_b b\leq x$ for all $b\in B$. Hence, for any finite subset 
$D\subseteq B$ and for any $c\in B$, we have by Corollary \ref{fsumato}  
that $c\oplus \bigoplus \{l_b b \mid b\in D\}$ exists. This 
yields that $\bigoplus \{l_b b \mid b\in D\}\leq c'$ and 
therefore $x\leq c'$ for all $c\in B$. 
Now, let $a\in A_x$. Then  $a\leq x\leq c'$ for all $c\in B$ i.e., $a\comp c$. 

We then have
$$
\begin{array}{r c l}
0\not = k^{x}_a a&=&k^{x}_a a \wedge x = k^{x}_a a \wedge \bigvee \{l_b b \mid b\in B\}\\
&=&\bigvee \{k^{x}_a a \wedge l_b b \mid b\in B\}= k^{x}_a a \wedge l_a a .
\end{array}
$$
The third equation follows from Statement \ref{lem:2.1}, (i) and 
the last equation follows from the fact that $a\not =b$, $a\comp b$ implies 
$k^{x}_a a \wedge l_b b=0$. Hence $k^{x}_a\leq l_a \leq k^{x}_a$ i.e., $a\in B$ and $l^{x}_a=k_a$. 
Therefore $A_x=B$ and $F_x=\{l_b b \mid b\in B\}$. The remaining part of the statement is evident.

\noindent{}(iii): Let $y\leq x$. Then $y\in \Mea(E)$, $A_y\subseteq A_x$ and  
$k^{y}_a \leq k^{x}_a $ for all $a\in A_y$. 
Recall that by \cite[Lemma 2.7 (i)]{PR6x} we 
know that $M$ is a bifull sub-lattice effect algebra of $E$. 
Since $M$ is atomic we have that 
$x=\bigoplus_M \{l_b b \mid b\in A^{M}_x\}=\bigoplus_E \{l_b b \mid b\in A^{M}_x\}$; here 
$A^{M}_x=\{a \mid \ a\ \text{an atom of}\ M,$ $ a\leq x\}$.
This immediately implies by (ii) that the sets 
$A^{M}_x$ and $A_x$ coincide. 
Therefore, $A_x\subseteq M$. Note also that $M$ is closed under arbitrary joins existing 
in $E$. Hence $y=\bigvee\{k^{y}_a a \mid \ a\leq y\}\in M$.

\noindent{}(iv): It follows immediately from (iii) and by  
$B(E)=\bigcap \{ M\subseteq E \mid 
M \ \text{is an atomic}$ $\text{block of}\ E\/\}$ (see \ref{lem:2.2}).

\noindent{}(v):  We have that $x=\bigoplus\{k^{x}_a a \mid a\ \text{an atom of}\ E, 
\ a\leq x\}$. 
Let $a\in A_x$. Then $a\leq x\leq \widehat{x}\in\Sh(E)$. Therefore 
 $n_a a\leq \widehat{x}$. Assume that $z\in\Sh(E)$, 
$n_a a\leq z$ for all $a\in A_x$. Then $k^{x}_a a\leq z$ 
 for all $a\in A_x$, i.e., $x\leq z$. This yields that 
 $\widehat{x}\leq z$, i.e. 
$\widehat{x}=\bigvee_{\Sh(E)}\{n_a a \mid\ a\ \text{an atom of}\ E, \ a\leq x\}$. 
By Statement \ref{shdisbifullcv}, (ii) we obtain that  
$\widehat{x}=\bigvee_{E}\{n_a a \mid\ a\ \text{an atom of}\ E, \ a\leq x\}$. 
Let $G\subseteq A_x$, $G$ finite. Then 
$\bigoplus \{n_a a \mid\ a\in G\}= \bigvee \{n_a a \mid\ a\in G\}\leq \widehat{x}$. 
Hence $\widehat{x}=\bigoplus\{n_a a \mid\ a\ \text{an atom of}\ E, \ a\leq x\}$. 

Further, we have 
$$
\begin{array}{@{}r c l}
\widehat{x}\ominus x&=& \left(\bigoplus\{n_{a} a \mid a\in A_x\}\right)\ominus %
\left(\bigoplus\{k^{x}_a a \mid a\in A_x\}\right)\\
\phantom{\text{\huge I}}&\geq&%
\left(\bigoplus\{n_{a} a \mid a\in A_x, a\not =b\}\oplus n_{b} b\right)\ominus %
\left(\bigoplus\{n_{a} a \mid a\in A_x, a\not =b\}\oplus k^{x}_b b\right)\\
\phantom{\text{\huge I}}&=& \left(n_{b} -k^{x}_b\right)b
\end{array}
$$
\noindent{}for all $b\in A_x$.
Now, let $z\in E$ such that $z\geq \left(n_{b} -k^{x}_b\right)b$ 
for all $b\in A_x$. Then also 
$z\wedge (\widehat{x}\ominus x)\geq \left(n_{b} -k^{x}_b\right)b$. 
Hence $\left(z\wedge (\widehat{x}\ominus x)\right)\oplus k^{x}_b b\geq 
n_{b} b$ i.e., 
$\left(z\wedge (\widehat{x}\ominus x)\right)\oplus %
\bigoplus\{k^{x}_a a \mid a\in A_x\}\geq 
\bigvee\{n_{a}a\mid  a\in A_x\}$. This yields 
$$
\begin{array}{c}
z\geq z\wedge (\widehat{x}\ominus x)\geq 
 \left(\bigoplus\{n_{a} a \mid a\in A_x\}\right)\ominus %
\left(\bigoplus\{k^{x}_a a \mid a\in A_x\}\right)=\widehat{x}\ominus x.
\end{array} $$
Therefore $\widehat{x}\ominus x= \bigoplus\{(n_a-k^{x}_a) a \mid\ a\in A_x\}%
=\bigvee\{(n_a-k^{x}_a) a \mid\ a\in A_x\}$.

The equality  $\widehat{x}=\widehat{\widehat{x}\ominus x}$ 
follows from Lemma \ref{jpy2}, (ii).

\noindent{}(vi): Let $x=\bigoplus_{i=1}^{n} k_i a_i$. By (ii) we have that the only atoms 
below $x$ are $a_1, \dots,  a_n$. Hence 
$x=\bigoplus_{i=1}^{n} k_i a_i=\bigvee_{i=1}^{n} k_i a_i$. From 
the proof of (iii) we know that any element of $[0, x]$ is of the 
form $\bigvee_{i=1}^{n} l_i a_i$ 
for uniquely determined natural numbers $0\leq l_i<n_{a_i}$, $1\leq i\leq n$ 
and conversely, for any system of natural numbers 
$0\leq l_i< n_{a_i}$, $1\leq i\leq n$, 
$\bigvee_{i=1}^{n} l_i a_i\in [0, x]$. This yields the required isomorphism between 
$[0, x]$ and $\prod _{i=1}^{n} [0, k_i a_i]$.
\end{proof}

Note that  Theorem \ref{popismeager} (ii), (iv) immediately yields that 
the set of meager (finite meager)
elements of an atomic Archime\-dean  lattice
effect algebra is a dual of a weak implication algebra introduced in \cite{chhak}.

Motivated by \cite[Proposition 15]{jenca} we have the following proposition.

{\renewcommand{\labelenumi}{{\normalfont  (\roman{enumi})}}
\begin{proposition}\label{mvme}
Let $E$ be an atomic Archimedean MV-effect algebra. Then:
\begin{enumerate}
\settowidth{\leftmargin}{(iiiii)}
\settowidth{\labelwidth}{(iii)}
\settowidth{\itemindent}{(ii)}
\item Let $x\in \Mea(E)$ and  $y\in E$ such that $x\wedge y=0$ and 
$\widehat{x}$ exists.  Then 
$\widehat{x}\wedge y=0$.
\item $\Mea(E)$ is a $\bigvee$-bifull sub-poset of $E$.
\item $\Mea(E)$ is a lattice ideal  of $E$.
\end{enumerate}
\end{proposition}}
\begin{proof} (i): As in Theorem \ref{popismeager} let us put 
$A_x=\{a \mid \ a\ \text{an atom of}\ E,$ $ a\leq x\}$. 
Evidently, $a\wedge y=0$ and $y\leq a'$ for all $a\in A_x$. Therefore 
by  Statement \ref{lem:2.1}, (iii)
$n_a a\wedge y=0$  for all $a\in A_x$. 
Then Theorem \ref{popismeager}, (v) yields that 
$\widehat{x}\wedge y=\bigvee\{n_a a \mid\ a\in A_x\}\wedge y=%
\bigvee\{n_a a \wedge y\mid\ a\in A_x\}=0$.

\noindent{}(ii): Let $X\subseteq \Mea(E)$. Assume that 
$z=\bigvee_{\Mea(E)} X$ exists. Let $u\in E$ be an upper bound of $X$. Hence 
also $u\wedge z$ is an upper bound of $X$ and clearly $u\wedge z$ is meager. 
Therefore $z=u\wedge z\leq u$, i.e., $z=\bigvee_{E} X$. 

Now, assume that $z=\bigvee_{E} X$ exists. It is enough to check that 
$z\in \Mea(E)$. Let $t\in\Sh(E)$, $t\leq z$, $t\not =0$. Then 
there exists an atom $b\in E$ such that $b\leq t$. Let us put 
$k_b^{x}=\text{max}\{k \mid kb \leq x\}< n_b$ (since any $x\in X$ is meager) 
and $k_b=\text{max}\{k^{x}_b \mid x\in X\}< n_b$. 
Hence also $n_b b\leq t\leq z$ and 
$n_b b= n_b b\wedge \bigvee_{E} X=%
\bigvee_{E}\{ n_b b \wedge x\mid x\in X\}=%
 \bigvee_{E}\{ k_b^{x} b \mid x\in X\}=k_b b <n_b b$, a contradiction. 
 Hence $\widetilde{z}=0$ and $z\in \Mea(E)$.

\noindent{}(iii): It follows immediately from (ii) because $\Mea(E)$ is a 
downset in $E$ and $E$ is a lattice.
\end{proof}

Moreover we have 

{\renewcommand{\labelenumi}{{\normalfont  (\roman{enumi})}}
\begin{proposition}\label{meagerbe}Let $E$ be 
an atomic Archime\-dean  lattice effect algebra. 
Then 
\begin{enumerate}
\settowidth{\leftmargin}{(iiiii)}
\settowidth{\labelwidth}{(iii)}
\settowidth{\itemindent}{(ii)}
\item For all $X\subseteq \B(E)\cap \Mea(E)$, 
$\bigvee_E X$ exists iff \, $\bigvee_{\B(E)} X$ exists, in which case 
$\bigvee_E X=\bigvee_{\B(E)} X\in \Mea(E)$. 
\item $\B(E)\cap \Mea(E)$ is a $\bigvee$-bifull sub-poset of $E$.
\end{enumerate}
\end{proposition}}
\begin{proof} (i): Let $X\subseteq \B(E)\cap \Mea(E)$. Assume first that 
$z=\bigvee_{\B(E)} X$ exists. Any $x\in X$ is by Theorem \ref{popismeager} 
of the form $x=\bigvee_E\{k^{x}_a a \mid \ a\in A_x\}=%
\bigvee_{\B(E)}\{k^{x}_a a \mid \ a\in A_x\}$, 
$A_x\subseteq \B(E)\cap \Mea(E)$. Hence 
$z=\bigvee_{\B(E)}\{\bigvee_{\B(E)}\{k^{x}_a a \mid \ a\in A_x\} \mid 
x\in X\}$. Let us put $k_a=\text{max}\{k^{x}_a \mid x\in X\}< n_a$. 
Then $z=\bigvee_{\B(E)}\{k_a a \mid a\in A_x, x\in X\}$. 

First, we shall show that $z\in \Mea(E)$. Assume that there is 
$y\not =0$, $y\leq z$, $y\in \Sh(E)$. Then there is an atom $c\in E$ 
such that $c\leq y$ i.e., also $n_c c\leq y\leq z$. Either $c\in A_x$ 
for some $x\in X$ or $c\wedge a=0$ for all $a\in A_x, x\in X$. Let 
$c\in A_x$ for some $x\in X$. Then $n_c c\in B(E)$. Therefore 
$$
\begin{array}{r c l}
n_c c&=&n_c c \wedge z= %
n_c c \wedge \bigvee_{\B(E)}\{k_a a \mid a\in A_x, x\in X\}\\
&=&\bigvee_{\B(E)}\{n_c c \wedge k_a a \mid a\in A_x, x\in X\}%
=k_c c < n_c c,\phantom{\text{\huge I}}
\end{array}
$$
\noindent{}a contradiction. Now, let $c\wedge a=0$ for all $a\in A_x, x\in X$.
Then $c\comp a$ yields that $k_a a \leq (n_c c)'\in C(E)$. Hence 
$z\leq (n_c c)'$. But $n_c c=n_c c \wedge z\leq n_c c \wedge (n_c c)'=0$ and 
we have a contradiction again. Hence $z\in \Mea(E)$.

Now, let $u\in E$ be an upper bound of $X$. Then also 
$u\wedge z$ is an upper bound of $X$, $u\wedge z\leq z\in \B(E)\cap\Mea(E)$. 
From Theorem \ref{popismeager} we have that 
$u\wedge z\in \B(E)\cap\Mea(E)$. Hence $z\leq u\wedge z \leq u$ i.e., 
$z=\bigvee_{E} X$.

Now, assume that $\bigvee_{E} X$ exists. Then $\bigvee_{E} X=\bigvee_{\B(E)} X$ 
by Statement \ref{lem:2.1}, (i). Hence $\bigvee_{E} X\in \Mea(E)$ by the above argument.

\noindent{}(ii): It follows immediately from (i) because $\B(E)\cap \Mea(E)$ is a 
downset in $\B(E)$.
\end{proof}

The following statement is well known.

\begin{statement}\label{shdombe} \label{sharpbe} Let $E$ be a sharply dominating 
Archime\-dean atomic  lattice effect algebra.
Then
\begin{enumerate}
\item[\rm(i)]  {\rm{}\cite[Theorem 3.4]{wujunde}} For every $x\in E, x\neq 0$ 
there exists the unique $w_x\in \Sh(E)$, unique set of atoms $\{a_{\alpha}|\alpha\in\Lambda\}$ and
unique positive integers $k_{\alpha}\neq ord(a_{\alpha})$ such
that
$$x=w_x\oplus (\bigoplus\{k_{\alpha}a_{\alpha}|\alpha\in\Lambda\}).$$
We call such a decomposition the 
{\em basic decomposition} ({\em{}BDE} for short) of $x$.
\item[\rm(ii)] {\rm{}\cite[Theorem 3.2]{PR6x}} $B(E)$ is sharply dominating and  
for every $x\in B(E), x\neq 0$ 
there exists the unique $w_x\in
C(E)$, unique set $\{a_{\alpha}|\alpha\in\Lambda\}\subseteq B(E)$ of atoms of $E$  and
unique positive integers $k_{\alpha}\neq \ord(a_{\alpha})$ such
that
$$x=w_x\oplus (\bigoplus\{k_{\alpha}a_{\alpha}|\alpha\in\Lambda\}).$$
\item[\rm(iii)] {\rm{}\cite[Theorem 3.1]{PR6x}} Let $M\subseteq E$ be an 
atomic block of $E$. Then 
$M$ is sharply dominating and, for every $x\in M$, there exists 
BDE of $x$ in $M$ 
and it coincides with BDE of $x$ in $E$.
\end{enumerate}
\end{statement}

\begin{proposition}\label{meagerblock}Let $E$ be 
a sharply dominating atomic Archime\-dean  lattice effect algebra 
and let $B\subseteq E$  be an
atomic block of $E$. Then $\Mea(B)\subseteq \Mea(E)$.
\end{proposition}
\begin{proof} Let $x\in \Mea(B)$. Then by Theorem \ref{popismeager}, (ii) 
$x=0\oplus (\bigoplus_{B}\{k_{\alpha}a_{\alpha}|\alpha\in\Lambda\})$ 
for a set of atoms $\{a_{\alpha}|\alpha\in\Lambda\}$ of $B$ and
positive integers $k_{\alpha}\neq ord(a_{\alpha})$. Since  
$B$ is a bifull sub-lattice effect algebra of $E$ 
(see \cite[Lemma 2.7 (i)]{PR6x}) we obtain  that 
$x=0\oplus (\bigoplus_{E}\{k_{\alpha}a_{\alpha}|\alpha\in\Lambda\})$.
As $E$ is sharply dominating we have from Statement \ref{sharpbe}, (i) 
that 
$\widetilde{x}=0$ and hence  $x\in \Mea(E)$.
\end{proof}

Let us recall the following statement

\begin{statement}\label{lemkaol}{\rm\cite[Lemma 2]{kaolparie}}
 Let $(E;\oplus,0,1)$ be an Archimedean atomic lattice effect
algebra, $x, y\in E$, $x\comp y$. Then there is  an atomic block $B$ 
of $E$ such that $x, y\in B$. 
\end{statement}

Similarly to \cite[Proposition 23]{jenca} for 
complete lattice effect algebras we have now the following proposition.

{\renewcommand{\labelenumi}{{\normalfont  (\roman{enumi})}}
\begin{proposition} \label{shaco}
Let  $E$ be a sharply dominating atomic Archime\-dean  lattice
effect algebra and let $x, y\in \Mea(E)$.
Then 
\begin{enumerate}
\settowidth{\leftmargin}{(iiiii)}
\settowidth{\labelwidth}{(iii)}
\settowidth{\itemindent}{(ii)}
\item $x\comp y$ if and only if $x\vee y\in \Mea(E)$, 
\item If $x\oplus y$ exists and $x\oplus y=z\in\Sh(E)$  then 
$z=\widehat{x}=\widehat{y}$. 
\end{enumerate}
\end{proposition}}
\begin{proof} (i): Assume first that $x\comp y$. Then by Statement 
\ref{lemkaol} there is   an atomic block $B$ 
of $E$ such that $x, y\in B$. Since 
$B$ is an atomic Archimedean MV-effect algebra and 
$E$ is sharply dominating we have from Propositions 
\ref{mvme} and \ref{meagerblock} that 
$x\vee y\in \Mea(B)\subseteq \Mea(E)$.

Now, assume that $x\vee y\in \Mea(E)$. Then from Theorem 
\ref{popismeager}, (iii) we obtain that $[0, x\vee y]$ is 
an MV-effect algebra. This yields that $x\comp y$. 

(ii): It follows immediately from Lemma \ref{jmpy2}. 
\end{proof}

\begin{theorem}\label{popissharp} Let $E$ be 
a sharply dominating atomic Archime\-dean  lattice
effect algebra. Then for every $x\in E, x\neq 0$ there exists 
unique set of atoms $\{a_{\alpha}\mid \alpha\in\Lambda\}$ {\rm (}namely 
$\{a\in E \mid a\ \text{an atom of}\ E,$ $ a\leq x\ominus \widetilde{x}\}${\rm )} 
and unique positive integers $k_{\alpha}\neq n_{a_{\alpha}}$ {\rm (}namely 
$k_{\alpha}={\rm\text{\rm{}max}}\{ k\in {\mathbb N}\mid k{a_{\alpha}} 
\leq x\}${\rm )}  such that
$$
\begin{array}{r@{\,\,}c@{\,\,}l}
x&=&\widetilde{x}\oplus (\bigoplus\{k_{\alpha}a_{\alpha}\mid \alpha\in\Lambda\}).\\
& &\text{\phantom{\tiny and}}\\
\text{Moreover,}\quad& &\text{\phantom{\tiny and}}\\
x&=&\widetilde{x}\oplus (\bigvee\{k_{\alpha}a_{\alpha}\mid \alpha\in\Lambda\})=%
\widetilde{x}\vee (\bigvee\{k_{\alpha}a_{\alpha}\mid \alpha\in\Lambda\}),\phantom{\text{Moreover,}}\\
0&=&\widetilde{x}\wedge (\bigvee\{k_{\alpha}a_{\alpha}\mid \alpha\in\Lambda\})=%
\widetilde{x}\wedge (\bigoplus\{k_{\alpha}a_{\alpha}\mid \alpha\in\Lambda\}),\phantom{\text{\huge I}}\\
& &\text{\phantom{\tiny and}}\\
\widehat{x}&=&\bigvee\{\widetilde{x}\oplus n_{a_{\alpha}}a_{\alpha}\mid \alpha\in\Lambda\}=
\widetilde{x}\oplus (\bigvee\{n_{a_{\alpha}}a_{\alpha}\mid \alpha\in\Lambda\})\\
&=&\widetilde{x}\oplus (\bigoplus\{n_{a_{\alpha}}a_{\alpha}\mid \alpha\in\Lambda\})%
=\widetilde{x}\vee (\bigvee\{n_{a_{\alpha}}a_{\alpha}\mid \alpha\in\Lambda\}),%
\phantom{\text{\huge I}}\\
0&=&\widetilde{x}\wedge (\bigvee\{n_{a_{\alpha}}a_{\alpha}\mid \alpha\in\Lambda\})=%
\widetilde{x}\wedge (\bigoplus\{n_{a_{\alpha}}a_{\alpha}\mid \alpha\in\Lambda\}),\phantom{\text{\huge I}}\\
& &\text{\phantom{\tiny and}}\\
\widehat{x}&=&x\oplus%
\bigoplus\{(n_{a_{\alpha}}-k_{\alpha})a_{\alpha}\mid \alpha\in\Lambda\}\\
&=&x\oplus (\bigvee\{(n_{a_{\alpha}}-k_{\alpha})a_{\alpha}\mid \alpha\in\Lambda\}).%
\phantom{\text{\huge I}}
\end{array}$$
\end{theorem}
\begin{proof} The first part of the statement follows immediately 
from Statement \ref{sharpbe}, (i)  
and Theorem \ref{popismeager}. 
Let us show the second, third and fourth parts. 

We have by Theorem \ref{popismeager} that 
$x\ominus \widetilde{x}=\bigoplus\{k_{\alpha}a_{\alpha}\mid \alpha\in\Lambda\}%
=\bigvee\{k_{\alpha}a_{\alpha}\mid \alpha\in\Lambda\}$. 
Hence by Lemma \ref{jpy2}, (i)
$(x\ominus \widetilde{x})\wedge \widetilde{x} =0$ and 
$\widetilde{x}\vee (x\ominus \widetilde{x})= %
\widetilde{x}\oplus (x\ominus \widetilde{x})=x$. Since 
$\widehat{x\ominus \widetilde{x}}$ exists we have from 
Theorem \ref{popismeager}, (v) that 
$\widehat{x\ominus \widetilde{x}}=%
\bigvee\{n_{a_{\alpha}}a_{\alpha}\mid \alpha\in\Lambda\}$. 
Therefore by by Lemma \ref{jpy2}, (ii) we have that 
$\widehat{x}=\widetilde{x}\oplus \widehat{x\ominus \widetilde{x}}=%
\widetilde{x}\vee \widehat{x\ominus \widetilde{x}}$  and 
$\widetilde{x}\wedge \widehat{x\ominus \widetilde{x}}=0$. 
 Moreover,  by Statement \ref{lem:2.1}, (vii) %
$\widetilde{x}\oplus (\bigvee\{n_{a_{\alpha}}a_{\alpha}\mid \alpha\in\Lambda\})=%
\bigvee\{\widetilde{x}\oplus n_{a_{\alpha}}a_{\alpha}\mid \alpha\in\Lambda\}$.

The fourth part follows immediately from the precedings parts. Namely, 
by Theorem \ref{popismeager}, (v)
$$
\begin{array}{@{}r@{\, }c@{\, }l}
\widehat{x}\ominus x&=&(\widehat{x}\ominus \widetilde{x})\ominus %
({x}\ominus \widetilde{x})= %
\bigoplus\{(n_{a_{\alpha}}-k_{\alpha})a_{\alpha}\mid \alpha\in\Lambda\}\\
&=&\bigvee\{(n_{a_{\alpha}}-k_{\alpha})a_{\alpha}\mid \alpha\in\Lambda\}.%
\phantom{\text{\huge I}}
\end{array}
$$
\end{proof}

{\renewcommand{\labelenumi}{{\normalfont  (\roman{enumi})}}
\begin{theorem}\label{ccnied} Let $E$ be a 
sharply dominating  atomic Archimedean lattice effect algebra. 
Then  the following conditions 
are equivalent: 
\begin{enumerate}
\settowidth{\leftmargin}{(iiiii)}
\settowidth{\labelwidth}{(iii)}
\settowidth{\itemindent}{(ii)}
\item $\B(E)$ is bifull in $E$. 
\item $\C(E)$ is bifull in $E$.
\end{enumerate}
\end{theorem}
\begin{proof}(i) $\implik$ (ii): Note that from  Statement \ref{sharpbe}, (ii) 
we know that $\B(E)$ is sharply dominating. 
Hence  by Statement \ref{shdisbifullcv} we obtain that 
$\C(E)=\Sh(\B(E))$ is bifull in $\B(E)$. Since $\B(E)$ is bifull in $E$ we have that 
$\C(E)$ is bifull in $E$.

\noindent{}(ii) $\implik$ (i): Let $S\subseteq \B(E)$ and 
$x=\bigvee_{\B(E)} S$ exists. 
Assume first that $x\in \C(E)$. Then
$$\begin{array}{@{\,}r@{\, }c@{\, }l}
x&=&\bigvee_{\B(E)} \{{s}\mid s\in S\}=%
\bigvee_{\B(E)} \{\widehat{{s}}\mid s\in S\}%
=\bigvee_{\C(E)} \{\widehat{{s}}\mid s\in S\}=%
\bigvee_{E} \{\widehat{s}\mid s\in S\}
\end{array}
$$
since $\C(E)$ is bifull in $E$ and $\widehat{{s}}\in \C(E)$.

Now, let $z\in E$, $z\geq s$ for all $s\in S$. Assume for a moment that 
$z\not\geq x$. Then $z\wedge x < x$ which yields that 
there is an atom $c\in E$ such that $c\leq x\ominus (x\wedge z)$. 
Then $c=c\wedge x= c\wedge \bigvee_{E} \{\widehat{{s}}\mid s\in S\}=
\bigvee_{E} \{c\wedge \widehat{{s}}\mid s\in S\}$ since $c\comp s$ and hence 
$x\comp \widehat{{s}}$ for all $a\in S$. Therefore, there 
is an element $s\in S$ such that 
$c\wedge \widehat{{s}}=c$ i.e. 
$c\leq \widehat{{s}}$.

By Theorem \ref{popissharp} and from Statement \ref{sharpbe}, (ii) 
we have that there exists the unique set 
$\{a_{\alpha}\mid \alpha\in\Lambda\}\subseteq \B(E)$ of atoms of 
$E$  and
unique positive integers $k_{\alpha}\neq n_{a_{\alpha}}$ such
that
$$
\begin{array}{@{}r c l}
s&=&\widetilde{s}\oplus %
(\bigoplus_{E}\{k_{\alpha}a_{\alpha}\mid \alpha\in\Lambda\})=%
\widetilde{s}\oplus (\bigvee_{E}\{k_{\alpha}a_{\alpha}\mid \alpha\in\Lambda\}),\\
& &\text{\phantom{\tiny and}}\\
\widehat{s}&=&%
\widetilde{s}\oplus (\bigvee_E\{n_{a_{\alpha}}a_{\alpha}\mid \alpha\in\Lambda\})%
=\widetilde{s}\vee (\bigvee_E\{n_{a_{\alpha}}a_{\alpha}\mid \alpha\in\Lambda\}).
\end{array}$$

This gives rise to 
$$
\begin{array}{r c l}
c&=& c \wedge \widehat{s}%
= (c \wedge \widetilde{s})\vee %
(c\wedge (\bigvee_E\{n_{a_{\alpha}}a_{\alpha}\mid \alpha\in\Lambda\}))
\end{array}
$$

\noindent{}since $c\comp \widetilde{s}\in\B(E)$ and 
$c\comp a_{\alpha}\in \B(E)$ for all $\alpha\in\Lambda$.

Assume first that $c=c \wedge \widetilde{s}$. Then from 
Statement \ref{lem:2.1}, (v) we have 
$n_c c \leq \widetilde{s}\leq z\wedge x$, a contradiction with 
$c\leq x\ominus (x\wedge z)$. So we obtain that 
$$
\begin{array}{r c l}
c&=&c\wedge (\bigvee_E\{n_{a_{\alpha}}a_{\alpha}\mid \alpha\in\Lambda\})%
=\bigvee_E\{c\wedge n_{a_{\alpha}}a_{\alpha}\mid \alpha\in\Lambda\}.
\end{array}
$$

Hence there is an atom $a_{\alpha}$ of $E$, $a_{\alpha}\in \B(E)$, 
$\alpha\in\Lambda$ such that $c=c\wedge n_{a_{\alpha}}a_{\alpha}$. 
Assume for a moment that $c\wedge a_{\alpha}=0$. We have that $c\comp a_{\alpha}$ 
i.e., $c\oplus a_{\alpha}$  exists. By Statement \ref{lem:2.1}, (iii) 
we have that  $c\wedge n_{a_{\alpha}}a_{\alpha}=0$, 
a contradiction. So we have shown that 
$c=a_{\alpha}\in \B(E)$. But $x\wedge c\leq x\ominus c < x$, $x\ominus c\in \B(E)$ 
and $x\ominus c$ is an upper bound of $S$, 
a contradiction with $x=\bigvee_{\B(E)} \{s\mid s\in S\}$.
Therefore $z\geq x$ and hence $\bigvee_{\B(E)} S=\bigvee_{E} S$.

Now, let us assume that $\bigvee_{\B(E)} S=x\in \B(E)$. Then we have that 
$$
\begin{array}{r@{\,}c@{\,}l}
\widehat{x}&=&x\oplus (\widehat{x}\ominus x)= 
\bigvee_{\B(E)} S \oplus (\widehat{x}\ominus x) = %
\bigvee_{\B(E)} \{s \oplus (\widehat{x}\ominus x)\mid s\in S\}\in\C(E). 
\end{array}
$$
Therefore by above considerations also $\widehat{x}= %
\bigvee_{E} \{s \oplus (\widehat{x}\ominus x)\mid s\in S\}$. This 
and Statement \ref{lem:2.1}, (iv) 
yield
$$
\begin{array}{r@{\,}c@{\,}l}
{x}&=&\widehat{x}\ominus (\widehat{x}\ominus x)=%
\left(\bigvee_{E} \{s \oplus (\widehat{x}\ominus x)\mid s\in S\}\right)%
\ominus (\widehat{x}\ominus x)=\bigvee_{E} S.%
\end{array}
$$

Conversely, let $S\subseteq \B(E)$ and 
$x=\bigvee_{E} S$ exists. Then by Statement \ref{lem:2.1}, (ii) we get 
that $x=\bigvee_{\B(E)} S$.
\end{proof}

{\renewcommand{\labelenumi}{{\normalfont  (\roman{enumi})}}
\begin{theorem}\label{ccatomic} Let $E$ be a sharply 
dominating  atomic Archimedean lattice effect algebra. Then  
the following conditions 
are equivalent: 
\begin{enumerate}
\settowidth{\leftmargin}{(iiiii)}
\settowidth{\labelwidth}{(iii)}
\settowidth{\itemindent}{(ii)}
\item $\B(E)$ is atomic. 
\item $\C(E)$ is atomic.
\end{enumerate}
\end{theorem}
\begin{proof}(i) $\implik$ (ii): Let $c\in \C(E)\subseteq \B(E)$, $c\not=0$. Then there 
is an atom $a\in \B(E)$ such that $a\leq c$. Therefore by 
Statement \ref{lem:2.1}, (v) $n_a a\leq c$ and $n_{a} a\in \C(E)=\B(E)\cap \Sh(E)$ 
since $\B(E)$ is a sub-lattice effect algebra of $E$. It follows that 
$[0, n_a a]=\{0, a, \dots, n_a a\}\subseteq B(E)$, as 
for every atom $b$ of $E$, $b\not =a$ we have $b \comp a$, 
which gives that $b\wedge n_a a=0$, by Statement \ref{lem:2.1}, (iii) and this 
yields by Statement \ref{lem:2.0}, (i) that 
any element below $n_a a$ is of the form $ka$, $0\leq k\leq n_a a$. 
Hence 
$\{0, a, 2a, \dots, n_a a\}\cap \C(E)=\{0, n_a a\}$. 
This yields that $n_a a$ is an atom of 
$\C(E)$ below $c$.

\noindent{}(ii) $\implik$ (i): Let $x\in \B(E)$, $x\not=0$.  If 
$x\not\in \Sh(E)$ then by  Statement \ref{sharpbe}, (ii) 
there is an atom $a\in \B(E)$ such that 
$a\leq x\ominus \widetilde{x}\leq x$. So let us assume that $x\in \Sh(E)\cap \B(E)=\C(E)$. 
Then there is by (ii) an atom $c$ from  $\C(E)$, $c\leq x$. Assume that there is an element 
$y\in \B(E)$ such that $y< c$. Then we have the following possibilities: 
{\renewcommand{\labelenumii}{{\normalfont  (\arabic{enumii})}}
\begin{enumerate}
\settowidth{\leftmargin}{(iiiii)}
\settowidth{\labelwidth}{(iii)}
\settowidth{\itemindent}{(ii)}
\item $y\not\in \Sh(E)$ and by the above argument there is an atom $a\in \B(E)$ such that 
$a\leq y < c \leq x$. Otherwise we have
\item $y\in \Sh(E)\cap \B(E)=\C(E)$ which implies that $y=0$.
\end{enumerate}}

Hence we obtain that $\B(E)$ is atomic. 
\end{proof}

\begin{corollary} \label{finc}Let $E$ be a sharply 
dominating atomic Archimedean  lattice effect algebra 
with a finite center $\C(E)$. Then 
$\B(E)$ is atomic and bifull in $E$. 
\end{corollary}

\section{Triple Representation Theorem for 
sharply dominating atomic\\ Archimedean lattice effect algebras}

In what follows $E$ will be always a sharply dominating atomic Archimedean 
lattice effect algebra. Then $\Sh(E)$ is a sub-lattice effect algebra 
of $E$ and $\Mea(E)$ equipped with a partial 
operation $\oplus_{\Mea(E)}$ which is defined, for all $x, y\in \Mea(E)$, 
by $x\oplus_{\Mea(E)} y$ exists
if and only if $x\oplus_E y$  exists and $x\oplus_E y\in \Mea(E)$ in which 
case $x\oplus_{\Mea(E)} y=x\oplus_E y$ is a generalized  effect algebra. 
Recall only that, for any meager atom $a\in E$, we have that 
$\ord_{\Mea(E)}(a)=\ord_{E}(a)-1$. We are therefore able to reconstruct 
the isotropic index in $E$ of any atom from $\Mea(E)$.
Moreover, we have a map 
$h:\Sh(E)\to 2^{\Mea(E)}$ that is given by 
$h(s)=\{x\in \Mea(E) \mid x\leq s\}$. As in \cite{jenca} for complete 
lattice effect algebras we will 
prove the following theorem. 
\medskip

\noindent{\bfseries Triple Representation Theorem}
 {\em The triple $(\Sh(E), \Mea(E), h)$ characterizes $E$ up to isomorphism.}
\medskip

We have to  construct an isomorphic copy of the original effect algebra
$E$ from the triple $(\Sh(E), \Mea(E), h)$. To do this we will first
construct  the
following mappings in terms of the triple.

\begin{itemize}
\item[(M1)] The mapping\quad $\widehat{\phantom{x}}:\Mea(E) \to \Sh(E)$.
\item[(M2)] For every $s\in \Sh(E)$, a mapping 
$\pi_{s}:\Mea(E) \to h(s)$, which is given by 
$\pi_{s}(x)=x\wedge_{E} s$.
\item[(M3)] The mapping\, $R:\Mea(E) \to \Mea(E)$ 
given by $R(x)=\widehat{{x}}\ominus_E x$.
\end{itemize}
\begin{itemize}
\item[(M4)] The partial mapping\, \mbox{${S}:\Mea(E) \times \Mea(E)\to \Sh(E)$} 
given by $S(x, y)$ is defined if and only if 
the set ${{\mathscr S}}(x, y)=\{z\in \Sh(E)\mid %
z=(z\wedge x)\oplus_E (z\wedge y)\}$ has a top element 
$z_0\in {\mathscr S}(x, y)$ 
in which case $S(x, y)=z_0$.
\end{itemize}

Since $E$ is sharply dominating and $\Sh(E)$ is bifull in $E$ 
we have that, for all $x\in \Mea(E)$, 
$$\widehat{x}=\bigwedge_{E}\{s\in \Sh(E) \mid x\in h(s)\}=%
\bigwedge_{\Sh(E)}\{s\in \Sh(E) \mid x\in h(s)\}.$$ 

Similarly, for all $s\in \Sh(E)$ and for all $x\in \Mea(E)$, 
$x\wedge_{E} s\in \Mea(E)$. Hence 
$$
\begin{array}{@{}r@{\,}c@{\,}l}
\pi_{s}(x)&=&x\wedge_{E} s=%
\bigvee_{E}\{y\in E \mid y\leq x, y\leq s\}\\
\phantom{\text{\huge I}}&=&\bigvee_{E}\{y\in \Mea(E) \mid y\leq x, y\in h(s)\}
=\bigvee_{M(E)}\{y\in \Mea(E) \mid y\leq x, y\in h(s)\}.
\end{array}
$$

Now, let us construct the mapping  $R$. Let $x\in  \Mea(E)$. If $x=0$ we put 
$R(x)=0$. Let $x\not =0$.
As before let us denote  by 
$A_x=\{a \mid \ a\ \text{an atom of}\ E,$ $ a\leq x\}%
=\{a \mid \ a\ \text{an atom of}\ \Mea(E),$ $ a\leq x\}\not =\emptyset$ 
and, for any $a\in A_x$, 
we shall put $k^{x}_a=\text{max}\{ k\in {\mathbb N}\mid ka \leq x\}$ 
and $n_a=\ord_{\Mea(E)}(a)+1$. Hence 
$1\leq k^{x}_a\leq \ord_{\Mea(E)}(a)$. Therefore 
$\{k^{x}_a a \mid \ a\in A_x\}\cup %
\{(n_a-k^{x}_a) a \mid\ a\in A_x\}\subseteq \Mea(E)$.
We know from Theorem \ref{popismeager}, (ii) and (v) that 
$x=\bigvee_{E}\{k^{x}_a a \mid \ a\in A_x\}=%
\bigvee_{\Mea(E)}\{k^{x}_a a \mid \ a\in A_x\}$, 
$\widehat{x}\ominus x=\bigvee_{E}\{(n_a-k^{x}_a) a \mid\ a\in A_x\}%
=\bigvee_{\Mea(E)}\{(n_a-k^{x}_a) a \mid\ a\in A_x\}\not = 0$, 
$\widehat{x}\ominus x\in  \Mea(E)$.

What remains is the partial mapping $S$. Let $x, y\in \Mea(E)$. 
By Lemma \ref{shaco}, (ii) ${\mathscr S}(x, y)=%
\{z\in \Sh(E) \mid %
z=(z\wedge x)\oplus_E (z\wedge y)\}=\{z\in \Sh(E)\mid %
z=\widehat{\pi_z(x)}\ \text{and}\ R(\pi_z(x))=\pi_z(y)\}$.
Hence whether $S(x, y)$ is defined or not we are able to decide 
in terms of the triple. Since the eventual top element $z_0$ of 
${\mathscr S}(x, y)$ is in $\Sh(E)$ our definition of ${S}(x, y)$ 
is correct.

\begin{lemma}\label{pommeag} Let $E$ be a sharply dominating atomic Archimedean 
lattice effect algebra, $x, y\in \Mea(E)$. Then $x\oplus_{E} y$ exists 
in $E$
iff $S(x, y)$ is defined in terms of the triple $(\Sh(E), \Mea(E), h)$ and 
$(x\ominus_{\Mea(E)} (S(x, y)\wedge x))%
\oplus_{\Mea(E)} (y\ominus_{\Mea(E)} (S(x, y)\wedge y))$ 
exists in $\Mea(E)$ such that %
$(x\ominus_{\Mea(E)} (S(x, y)\wedge x))%
\oplus_{\Mea(E)} (y\ominus_{\Mea(E)} (S(x, y)\wedge y))%
\in h(S(x, y)')$. Moreover, in 
that case
$$x\oplus_{E} y=\underbrace{S(x, y)}_{\in \Sh(E)}\oplus_{E} %
(\underbrace{(x\ominus_{\Mea(E)} (S(x, y)\wedge x))%
\oplus_{\Mea(E)} (y\ominus_{\Mea(E)} (S(x, y)\wedge y))}_{\in \Mea(E)}).$$
\end{lemma}
\begin{proof} Assume first that $x\oplus_{E} y$ exists 
in $E$ and let us put $z=x\oplus_{E} y$. Then 
$z=z_S\oplus_{E}  z_M$ such that $z_S\in \Sh(E)$ and 
$z_M\in \Mea(E)$ is BDE of $z$ in $E$. Since $x\comp y$ by Statement 
\ref{lemkaol} there is  an atomic block $B$ 
of $E$ such that $x, y, z\in B$. We know from 
Statement \ref{shdombe}, (iii) that $B$ is sharply dominating 
and BDE of $z\in B$ in $B$ and BDE of $z$ in $E$ coincide. This yields 
that $z_S,  z_M\in B$. Therefore $z_S\in \C(B)$ and by Statement 
\ref{lem:2.1}, (viii) we have that 
 $z_S= z_S \wedge (x\oplus_{E} y)= z_S \wedge (x\oplus_{B} y)= 
 (z_S \wedge x)\oplus_{B} (z_S \wedge y)= %
(z_S \wedge x)\oplus_{E} (z_S \wedge y)$. Hence $z_S\in {\mathscr S}(x, y)$. 
Now, assume that $u\in {\mathscr S}(x, y)$. Then 
$u=(u \wedge x)\oplus_{E} (u \wedge y)\leq x\oplus_{E} y$. Since 
$u\in \Sh(E)$ we have that $u\leq z_S$, i.e., $z_S$ is the top 
element of ${\mathscr S}(x, y)$. Moreover, we have 
$$
\begin{array}{@{}r@{\,}c@{\,}l}
z_S\oplus_E z_M&=&x\oplus_{E} y\\
&=&((S(x, y)\wedge x)\oplus_{E} %
(x\ominus_{E} (S(x, y)\wedge x)))%
\oplus_E \\
& &((S(x, y)\wedge y)\oplus_{E}
 (y\ominus_{E} (S(x, y)\wedge y)))\\
&=&{S(x, y)}\oplus_{E} %
((x\ominus_{\Mea(E)} (S(x, y)\wedge x))%
\oplus_{E}
(y\ominus_{\Mea(E)} (S(x, y)\wedge y))).
 \end{array}
$$
Because $z_S=S(x, y)$ it follows that $z_M=(x\ominus_{\Mea(E)} (S(x, y)\wedge x))%
\oplus_{E} (y\ominus_{\Mea(E)} (S(x, y)\wedge y))$, i.e., 
$z_M=(x\ominus_{\Mea(E)} (S(x, y)\wedge x))%
\oplus_{\Mea(E)} (y\ominus_{\Mea(E)} (S(x, y)\wedge y))$ and evidently 
$z_M\in h(z_S')$.

Conversely, let us assume that $S(x, y)$ is defined in terms of  $(\Sh(E), \Mea(E), h)$,  
$(x\ominus_{\Mea(E)} (S(x, y)\wedge x))%
\oplus_{\Mea(E)} (y\ominus_{\Mea(E)} (S(x, y)\wedge y))$ 
exists in $\Mea(E)$ and %
$(x\ominus_{\Mea(E)} (S(x, y)\wedge x))%
\oplus_{\Mea(E)} (y\ominus_{\Mea(E)} (S(x, y)\wedge y))%
\in h(S(x, y)')$. Then 
$(x\ominus_{\Mea(E)} (S(x, y)\wedge x))%
\oplus_{\Mea(E)} (y\ominus_{\Mea(E)} (S(x, y)\wedge y))%
\leq S(x, y)'$, i.e., 
$$
\begin{array}{r c l}
&&S(x, y)\oplus_{E}%
((x\ominus_{\Mea(E)} (S(x, y)\wedge x))%
\oplus_{\Mea(E)} (y\ominus_{\Mea(E)} (S(x, y)\wedge y)))\\
&=&((S(x, y)\wedge x)\oplus_{E} (S(x, y)\wedge y)) \oplus_{E}\\%
&&((x\ominus_{E} (S(x, y)\wedge x))%
\oplus_{E} (y\ominus_{E} (S(x, y)\wedge y)))=x\oplus_{E} y\\
\end{array}
$$
\noindent{}is defined.
\end{proof}

{\renewcommand{\labelenumi}{{\normalfont  (\roman{enumi})}}
\begin{theorem}\label{tripletheor}
Let E be a sharply dominating atomic Archimedean lattice effect algebra. 
Let $\Tea(E)$ be a subset
of $\Sh(E) \times \Mea(E)$ given by
$$\Tea(E) =\{(z_S, z_M)\in \Sh(E) \times \Mea(E) \mid z_M \in  h(z_S')\}.$$
Equip $\Tea(E)$ with a partial binary operation $\oplus_{\Tea(E)}$ with 
$(x_S, x_M) \oplus_{\Tea(E)} (y_S, y_M)$ is defined if and  only if 
\begin{enumerate}
\item $S(x_M, y_M)$ is defined,
\item $z_S=x_S\oplus_{\Sh(E)} y_S \oplus_{\Sh(E)} S(x_M, y_M)$ is defined,
\item $z_M=(x_M\ominus_{\Mea(E)} (S(x_M, y_M)\wedge x_M))%
\oplus_{\Mea(E)} (y_M\ominus_{\Mea(E)} (S(x_M, y_M)\wedge y_M))$ is defined,
\item $z_M\in h(z_S')$. 
\end{enumerate}
In this case $(z_S, z_M)=(x_S, x_M) \oplus_{\Tea(E)} (y_S, y_M)$.
Let\/ $0_{\Tea(E)}=(0_E,0_E)$ and\/ $1_{\Tea(E)}=(1_E,0_E)$. Then 
$\Tea(E)=(\Tea(E), \oplus_{\Tea(E)}, 0_{\Tea(E)},$ $1_{\Tea(E)})$ 
is an effect algebra and the mapping 
$\varphi:E\to \Tea(E)$  given by 
$\varphi(x) = (\widetilde{x}, x\ominus_{E} \widetilde{x})$   
is an isomorphism of effect algebras.
\end{theorem}}
\begin{proof} Evidently, $\varphi$ is correctly defined since, 
for any $x\in E$, we have that $x=\widetilde{x}\oplus_{E} (x\ominus \widetilde{x})=%
x_S\oplus_{E} x_M$, $x_S\in \Sh(E)$ and $x_M\in \Mea(E)$. Hence 
$\varphi(x) = (x_S, x_M)\in \Sh(E) \times \Mea(E)$ and $x_M\in h(x_S')$. Let us check that $\varphi$ is 
bijective. Assume first that $x, y\in E$ such that $\varphi(x) =\varphi(y)$. 
We have $x=\widetilde{x}\oplus_{E} (x\ominus_{E} \widetilde{x})=%
\widetilde{y}\oplus_{E} (y\ominus_{E} \widetilde{y})=y$. Hence 
$\varphi$ is injective. Let $(x_S, x_M)\in \Sh(E) \times \Mea(E)$ and 
$x_M\in h(x_S')$. This yields that $x=x_S\oplus_{E} x_M$ exists and 
evidently by Lemma \ref{jpy2}, (i) $\widetilde{x}=x_S$ and $x\ominus_{E} \widetilde{x}=x_M$.
It follows that $\varphi$ is surjective. Moreover, 
$\varphi(0_{E})=(0_E,0_E)=0_{\Tea(E)}$ and\/ $\varphi(1_{E})=(1_E,0_E)=1_{\Tea(E)}$.

Now, let us check that, for all $x, y\in E$, $x\oplus_E y$ 
is defined iff $\varphi(x)\oplus_{\Tea(E)} \varphi(y)$ is defined 
in which case $\varphi(x\oplus_E y)=\varphi(x)\oplus_{\Tea(E)} \varphi(y)$. 
For any $x, y, z, u\in E$ we obtain 
\begin{center}
\begin{tabular}{@{}c r c l}
\multicolumn{4}{l}{$z=x\oplus_E y$  $\iff$
$z=(\widetilde{x}\oplus_E (x\ominus_{E} \widetilde{x}))%
\oplus_E (\widetilde{y}\oplus_E (y\ominus_{E} \widetilde{y}))$} \\
$\iff$&\multicolumn{3}{l}{$z=(\widetilde{x}\oplus_E \widetilde{y})%
\oplus_E ((x\ominus_{E} \widetilde{x})\oplus_E (y\ominus_{E} \widetilde{y}))$ \ $\iff$ \ %
by Lemma \ref{pommeag}}\\
\multicolumn{4}{l}{$u=S(x\ominus_{E} \widetilde{x}, y\ominus_{E} \widetilde{y})$  and}\\ %
\multicolumn{4}{l}{$\begin{array}{@{}r@{}c@{}l}
z=(\widetilde{x}\oplus_E \widetilde{y})%
\oplus_E (u&\oplus_{E}&((x\ominus_{E} \widetilde{x})\ominus_{E} (u\wedge (x\ominus_{E} \widetilde{x})))\\
&\oplus_{E}&((y\ominus_{E} \widetilde{y})\ominus_{E} (u\wedge (y\ominus_{E} \widetilde{y}))))\ 
\end{array}$ }\\
$\iff$&\multicolumn{3}{l}{$u=S(x\ominus_{E} \widetilde{x}, y\ominus_{E} \widetilde{y})$  and}\\ %
\multicolumn{4}{l}{$\begin{array}{r@{}c@{}l}
z=(\widetilde{x}\oplus_E \widetilde{y}%
\oplus_E u)&\oplus_{E}&(((x\ominus_{E} \widetilde{x})\ominus_{E} (u\wedge (x\ominus_{E} \widetilde{x})))\\
&\oplus_{E}&\phantom{(}((y\ominus_{E} \widetilde{y})\ominus_{E} (u\wedge (y\ominus_{E} \widetilde{y}))))\ 
\end{array}$ }\\
$\iff$&\multicolumn{3}{l}{$u=S(x\ominus_{E} \widetilde{x}, y\ominus_{E} \widetilde{y})$  and}\\ %
\multicolumn{4}{l}{$\begin{array}{r@{}c@{}l}
z=(\widetilde{x}\oplus_{\Sh(E)} \widetilde{y}%
&\oplus_{\Sh(E)}& u)\oplus_{E}(((x\ominus_{E} \widetilde{x})\ominus_{\Mea(E)} %
(u\wedge (x\ominus_{E} \widetilde{x})))\\
&\oplus_{\Mea(E)}&\phantom{(}((y\ominus_{E} \widetilde{y})\ominus_{\Mea(E)} %
(u\wedge (y\ominus_{E} \widetilde{y}))))\ 
\end{array}$ }\\
$\iff$&\multicolumn{3}{l}{$(\widetilde{x}, x\ominus_{E} \widetilde{x}) %
\oplus_{\Tea(E)} (\widetilde{y}, y\ominus_{E} \widetilde{y})$ is defined and}\\
\multicolumn{4}{l}{%
\begin{tabular}{@{}r@{}c@{}l}$\varphi(z)$&=&$(\widetilde{x}\oplus_{\Sh(E)}\widetilde{y}\oplus_{\Sh(E)}%
S(x\ominus_{E} \widetilde{x}, y\ominus_{E} \widetilde{y}), %
((x\ominus_{E} \widetilde{x}) \ominus (S(x\ominus_{E} \widetilde{x}, y\ominus_{E} \widetilde{y})%
\wedge $\\%
&&$(x\ominus_{E} \widetilde{x})))
\oplus_{\Mea(E)} ((y\ominus_{E} \widetilde{y}) \ominus %
(S(x\ominus_{E} \widetilde{x}, y\ominus_{E} \widetilde{y})%
\wedge (y\ominus_{E} \widetilde{y}))))$\\
&=&%
$(\widetilde{x}, x\ominus_{E} \widetilde{x}) %
\oplus_{\Tea(E)} (\widetilde{y}, y\ominus_{E} \widetilde{y})=\varphi(x)\oplus_{\Tea(E)} \varphi(y)$.
\end{tabular}}
\end{tabular}
\end{center}
Altogether, $\Tea(E)=(\Tea(E), \oplus_{\Tea(E)}, 0_{\Tea(E)},$ $1_{\Tea(E)})$ 
is an effect algebra and the mapping $\varphi:E\to \Tea(E)$     
is an isomorphism of effect algebras.
\end{proof}

The Triple Representation Theorem then follows immediately.

\begin{remark}\label{larem}{\rm Recall that our method may be also used in the case of 
complete lattice effect algebras as a substitute of the method from 
\cite{jenca} since we need only 
Lemma \ref{shaco}, (ii) and Lemma \ref{pommeag} to show that Theorem \ref{tripletheor} 
hols for complete lattice effect algebras. But to show that Lemma \ref{shaco}, (ii) 
and Lemma \ref{pommeag} 
hold for complete lattice effect algebras is an easy task.}
\end{remark}

Now, using Theorems \ref{ccatomic} and \ref{tripletheor} we can prove
the following Triple Representation Theorem for $\B(E)$
of sharply dominating atomic Archimedean lattice effect algebras $E$
with atomic center $\C(E)$.

{\renewcommand{\labelenumi}{{\normalfont  (\roman{enumi})}}
\begin{theorem}\label{vftripletheor}
Let E be a sharply dominating atomic Archimedean lattice effect algebra 
with atomic center $\C(E)$.
Let $\Tea(\B(E))$ be a subset
of $\C(E) \times (\Mea(E)\cap \B(E))$ given by
$$\Tea(\B(E)) =\{(z_C, z_{MB})\in \C(E) \times (\Mea(E)\cap \B(E)) 
\mid z_{MB} \in  h(z_C')\cap \B(E)\}.$$
Let us put 
$\oplus_{\Tea(\B(E))}:={\oplus_{\Tea(E)}}_{/\Tea(\B(E))\times \Tea(\B(E))}$ 
and let\/ $0_{\Tea(\B(E))}=(0_E,0_E)$ and\/ $1_{\Tea(\B(E))}=(1_E,0_E)$. Then 
$\Tea(\B(E))=(\Tea(\B(E)), \oplus_{\Tea(\B(E))}, 0_{\Tea(\B(E))},$ $1_{\Tea(\B(E))})$ 
is an effect algebra and the mapping 
$\varphi_{\B(E)}:\B(E)\to \Tea(\B(E))$  given by 
$\varphi_{\B(E)}=\varphi_{/\B(E)}$   
is an isomorphism of effect algebras.
\end{theorem}}
\begin{proof} Recall that from Statement \ref{sharpbe}, (ii) and Theorem \ref{ccatomic} we 
know that $\B(E)$ is a sharply dominating atomic Archimedean lattice effect algebra. 
Moreover, $\Sh(\B(E))=\C(E)$, $\Mea(\B(E))=\Mea(E)\cap \B(E)$,  
$h_{\B(E)}(c)=h(c)\cap \B(E)$ for all $c\in\C(E)$  and, for all 
$y\in \B(E)$, we have that by Statement \ref{sharpbe}, (ii) $\widetilde{y}\in \C(E)$ and  $\widehat{y}\in \C(E)$. Since $\B(E)$ and 
$\C(E)$ are sub-lattice effect algebras of $E$ we obtain that the mappings 
(M1)-(M4) for the triple $(\C(E), \Mea(\B(E)), h_{\B(E)})$ are natural restrictions of 
the mappings (M1)-(M4) for the triple $(\Sh(E), \Mea(E), h)$. Invoking Theorem 
\ref{tripletheor} we obtain the required statement.
\end{proof}


\section*{Acknowledgements}  J. Paseka gratefully acknowledges Financial Support 
of the  Ministry of Education of the Czech Republic
under the project MSM0021622409 and of Masaryk University under the grant 0964/2009.

\end{document}